\documentclass[12pt,reqno]{amsart}

\usepackage{amsfonts,amsmath,amsthm,bbm}
\usepackage{amssymb,epsfig}
\usepackage{cases}
\usepackage[usenames]{color}
\usepackage{enumerate} 

\usepackage{color} 
\definecolor{vert}{rgb}{0,0.6,0}

\theoremstyle{plain}
\newtheorem{thm}{Theorem}[section]

\newtheorem{defn}{Definition}

\newtheorem{lem}[thm]{Lemma}
\newtheorem{cor}[thm]{Corollary}
\newtheorem{prop}[thm]{Proposition}
\theoremstyle{remark}
\newtheorem{rem}{\bf{Remark}}
\numberwithin{equation}{section}

\newcommand{\N}{\mathbb{N}}

\newcommand{\R}{\mathbb{R}}

\newcommand{\T}{\mathbb{T}}
\newcommand{\Z}{\mathbb{Z}}

\newcommand{\cS}{\mathcal{S}}

\newcommand{\al}{\alpha}
\newcommand{\gam}{\gamma}

\newcommand{\ep}{\varepsilon}

\newcommand{\ol}{\overline}

\newcommand{\diag}{{\rm diag}}

\begin{document}

\title[Optimal rates of convergence]{Remarks on Optimal rates of convergence in periodic homogenization of linear elliptic equations in non-divergence form}

\author{Xiaoqin Guo}
\address[Xiaoqin Guo]
{
Department of Mathematics, 
University of Wisconsin Madison, Van Vleck hall, 480 Lincoln drive, Madison, WI 53706, USA}
\email{xguo@math.wisc.edu}

\author{Hung V. Tran}
\address[Hung V. Tran]
{
Department of Mathematics, 
University of Wisconsin Madison, Van Vleck hall, 480 Lincoln drive, Madison, WI 53706, USA}
\email{hung@math.wisc.edu}

\author{Yifeng Yu}
\address[Yifeng Yu]
{
Department of Mathematics, 
University of California, Irvine,
410G Rowland Hall, Irvine, CA 92697, USA}
\email{yyu1@math.uci.edu}

\dedicatory{Dedicated to Professor Hitoshi Ishii with our admiration}

\thanks{
The work of HT is partially supported by NSF grant DMS-1664424 and NSF CAREER award DMS-1843320.
}

\date{\today}
\keywords{Homogenization; periodic setting; linear non-divergence form elliptic equations; optimal rates of convergence}
\subjclass[2010]{
35B27, 
35B40, 
35D40, 
35J25, 
49L25 
}

\maketitle


\begin{abstract}
We study and characterize the optimal rates of convergence in periodic homogenization of linear  elliptic equations in non-divergence form. 
We obtain that the optimal rate of convergence is either $O(\ep)$ or $O(\ep^2)$ depending on the diffusion matrix $A$, source term $f$, and boundary data $g$.
Moreover, we show that the set of diffusion matrices $A$ that give optimal rate $O(\ep)$ is open and dense in the set of $C^{2,\al}$ periodic, symmetric, and positive definite matrices, which means that generically, the optimal rate is $O(\ep)$.
\end{abstract}

\section{Introduction}

In this paper, we are interested in studying and characterizing the optimal rates of convergence in periodic homogenization of linear  elliptic equations in non-divergence form.
Let $U \subset \R^n$ be a given bounded domain with smooth boundary.
The equation of our main interest is
\begin{equation}\label{D-ep}
\begin{cases}
-a_{ij}\left(\frac{x}{\ep}\right) u^\ep_{x_i x_j} = f(x) \qquad &\text{ in } U,\\
u^\ep = g \qquad &\text{ on } \partial U.
\end{cases}
\end{equation}
The matrix function $A(y)=(a_{ij})_{1\leq i,j \leq n} \in C^2(\R^n, \R^{n^2})$ is always assumed to be symmetric,  $\Z^n$-periodic, and positive definite for all $y \in \R^n$.
Denote by $\T^n = \R^n/\Z^n$ the flat $n$-dimensional torus, and $\cS^n_+$ the set of all real symmetric, positive definite matrices of size $n\times n$,
then we can also write that $A \in C^2\left(\T^n, \cS^n_+\right)$.
Assume  $f \in C^2\left(\ol U\right)$ and $g \in C^4(\partial U)$.
In this paper, we always use the Einstein summation convention.

\medskip
The homogenization problem \eqref{D-ep} was discussed in the classical books of Bensoussan, Lions, Papanicolaou \cite{BLP}, Jikov, Kozlov, Oleinik \cite{JKO}.
It is well-known that, as $\ep \to 0$, $u^\ep \to u$ uniformly on $\ol U$, where $u$ solves the following effective equation
\begin{equation}\label{D-0}
\begin{cases}
-\ol a_{ij} \, u_{x_i x_j} = f(x) \qquad &\text{ in } U,\\
u = g \qquad &\text{ on } \partial U.
\end{cases}
\end{equation}
Here, $\ol A=\{ \ol a_{ij}\}_{1\leq i,j \leq n}$ is the {\it effective matrix} with constant entries, which is determined as follows.
For each fixed $(k,l) \in \{1,\ldots, n\}^2$, consider the  solution $v^{kl}$ of the $(k,l)$-th cell problem
\begin{equation}\label{cell-kl}
-a_{ij}(y) v^{kl}_{y_i y_j}(y) - a_{kl}(y) = -\ol a_{kl}, \qquad y\in\T^n,
\end{equation}
where $\ol a_{kl} \in \R$ is the unique constant such that \eqref{cell-kl} has a solution $v^{kl}$.
In fact, $v^{kl}$ is unique up to an additive constant by the strong maximum principle.
Then, for a symmetric matrix $M$, the corresponding {\it corrector} is 
\begin{equation}\label{def: v}
v(y,M) = M_{kl} v^{kl}(y).
\end{equation}
It is clear that $v(y,M)$ solves
\[
-a_{ij}(M_{ij} + v_{y_i y_j}(y,M)) =  - \ol a_{ij} M_{ij} \qquad \text{ in } \T^n.
\]
On the other hand, $\ol A$ can also be determined through the corresponding invariant measure as follows.
Let $r \in C(\T^n)$ be the unique solution to
\begin{equation}\label{invariant meas}
\begin{cases}
-(a_{ij}(y) r(y))_{y_i y_j} = 0 \qquad \text{ in } \T^n,\\
r>0 \qquad \text{and} \qquad \int_{\T^n} r(y)\,dy=1.
\end{cases}
\end{equation}
We say that $r$ is the {\it invariant measure} of the matrix $A\in C^2(\T^n,\cS^n_+)$.
See Freidlin \cite{Fre}, Avellaneda, Lin \cite{AvL}, Evans \cite{Ev}, Engquist, Souganidis \cite{ES}.
Multiply \eqref{cell-kl} by $r$ and integrate to yield, for $1 \leq k,l\leq n$,
\[
\ol a_{kl}  = \int_{\T^n} a_{kl}(y) r(y)\,dy.
\]
And thus,
\[
\ol A  = \int_{\T^n} A(y)r(y)\,dy.
\]
Our main focus in this paper is to understand the optimal rate of convergence of $u^\ep$ to $u$, that is, the optimal upper bound of $\|u^\ep - u\|_{L^\infty(U)}$ as $\ep \to 0^+$.
Heuristically, by the two scale asymptotic expansions, around a given point $x_0 \in U$ with $M=D^2 u(x_0)$, one has the following expansion of $u^\ep$ for $x\approx x_0$:
\[
u^\ep(x) \approx u(x) + \ep^2 v\left(\frac{x}{\ep}, M\right) = u(x) + \ep^2 v^{kl}\left(\frac{x}{\ep}\right) u_{x_k x_l}(x_0).
\]
Naively, this suggests that $|u^\ep(x) - u(x)| \leq C \ep^2$ for $x \approx x_0$, and we might be able to obtain the rate of convergence $O(\ep^2)$ of $\|u^\ep - u\|_{L^\infty(U)}$ as $\ep \to 0^+$.
Of course, this $O(\ep^2)$ rate, if obtained, is optimal.

However, in the literature, only an $O(\ep)$ rate is known.
\begin{thm}[{\cite[Theorem 5.1, page 230]{BLP}}, {\cite[page 33]{JKO}}] \label{thm:prelim}
Assume that $f \in C^2\left(\ol U\right)$ and $g \in C^4(\partial U)$.
Then, there exists $C>0$ depending only on the ellipticity of $A$, $f$, $g$ such that
\begin{equation}\label{sub-optimal}
\|u^\ep - u\|_{L^\infty(U)} \leq C\ep.
\end{equation}
\end{thm}
In fact, using the doubling variable method in the theory of viscosity solutions, the regularity of $f$ and $g$ can be relaxed to allow $f \in C^1\left(\ol U \right)$ and $g \in C^3(\partial U)$.
In any case, the regularity of $f$ and $g$ is not the main concern in this paper.

Theorem \ref{thm:prelim} 
is well known in the literature.
See the classical books of Bensoussan, Lions, Papanicolaou \cite{BLP}, Jikov, Kozlov, Oleinik \cite{JKO},
and the review paper of Engquist, Souganidis \cite{ES}.
For the fully nonlinear settings, see Caffarelli, Souganidis \cite{CS}, Kim, Lee \cite{KL}.
For some numerical results in this direction, see  Froese, Oberman \cite{FrO}, Capdeboscq, Sprekeler, S\"uli \cite{CSS}.
This $O(\ep)$ rate of convergence is not known to be optimal or not.
Indeed, we have not yet been able to find any discussion on the optimality of $O(\ep)$ in the literature.

\medskip

Our paper provides satisfactory results to fill in this gap of knowledge in the literature.
It is one of our goals to clear out a misconception that the optimal rate of convergence is always $O(\ep^2)$, which is false in both periodic and random settings.
Surprisingly, we can show that ``almost all" matrices $A\in C^{2,\al}(\T^n,\cS_+^n)$ give an optimal rate of $O(\ep)$ for fixed $\al \in (0,1)$. 
To be more specific, such matrices form an open and dense set under the $C^{2,\al}(\T^n,\cS_+^n)$ topology. Furthermore, we provide examples where the optimal rate of homogenization is $O(\ep^2)$ when the diffusion matrix $A$, source term $f$, and boundary data $g$ satisfy special conditions.

Since the literature on homogenization is vast, we only give references on periodic homogenization of non-divergence form elliptic equations in the paper.
We describe our main results in the following section.

\subsection{Main results}


Let us now proceed to discuss about optimal rates of convergence of $u^\ep$ to $u$.
For $1\leq j, k, l \leq n$ fixed, denote by
\begin{equation}\label{eq:def-c_ijk}
c^{kl}_{j} =c^{kl}_{j}(A)= \int_{\T^n} a_{ij}(y) v^{kl}_{y_i}(y) r(y)\,dy.
\end{equation}
Note that $c^{kl}_j(A)$ depends only on $A$ but in a highly nonlinear way.

Set
\[
h(x) = c^{kl}_{j}  u_{x_j x_k x_l}(x) \qquad \text{ for all } x\in U.
\]
Let  $z$ be the solution to
\begin{equation}\label{z-0}
\begin{cases}
-\ol a_{ij} z_{x_i x_j}  =-h(x) \qquad &\text{ in } U,\\
z = 0 \qquad &\text{ on } \partial U.
\end{cases}
\end{equation}
Here is our first main result.

\begin{thm}\label{thm:main1}
Assume that $f \in C^3\left(\ol U\right)$ and $g \in C^5(\partial U)$.
Then, there exists $C>0$ depending only on the ellipticity of $A$, $f$, $g$ such that
\begin{equation}\label{u-ep-u-z}
\|u^\ep - u + 2 \ep z\|_{L^\infty(U)} \leq C\ep^2.
\end{equation}
In particular, the following claims hold.
\begin{itemize}
\item[(i)] If $h \equiv 0$, then $\|u^\ep - u\|_{L^\infty(U)} \leq C\ep^2$, and this rate of convergence $O(\ep^2)$ is optimal.

\item[(ii)] If $h\not\equiv 0$, then $\|u^\ep - u\|_{L^\infty(U)} \leq C\ep$, and this rate of convergence $O(\ep)$ is optimal.
\end{itemize}
\end{thm}

Of course, this theorem is rather abstract as we do not know precisely what $h$ and $z$ are in general.
It is clear that $h \equiv 0$ if and only if $z \equiv 0$, although 
$z$ depends not only on $h$ but also on the effective matrix $\ol A$.
In order to understand deeper \eqref{u-ep-u-z}, it is necessary to understand more about qualitative behavior of $h$.
In particular, it is important to know whether situations (i) and (ii) can happen or not.
It turns out that this is the case.

\begin{cor}\label{thm:main3}
If there exist $j,k,l \in \{1,\ldots, n\}$ such that 
\[
 C_{jkl}(A)=c^{kl}_j(A)+c^{jk}_l(A)+c^{jl}_k(A) \neq 0,
 \] 
then we can find $f,g$ such that situation {\rm (ii)} of Theorem \ref{thm:main1} holds true.
\end{cor}

Based on the above results, we see that $c^{kl}_j(A)$ for $1\leq j, k, l \leq n$ determine whether the optimal rate of convergence is $O(\ep^2)$ or $O(\ep)$ when no special conditions are imposed on $f$ and $g$.
This leads us to the following classification of
 matrices in $C^2\left(\T^n, \cS^n_+\right)$.

\begin{defn}
Let $A \in C^2\left(\T^n, \cS^n_+\right)$.
If $ C_{jkl}(A)=0$ for all $1\leq j, k, l \leq n$, then we say that $A$ is a $c$-good matrix.
Otherwise, $A$ is a $c$-bad matrix.
\end{defn}

Clearly, $c$-good matrices give optimal rate of convergence $O(\ep^2)$ as $h \equiv 0$.
And, for $c$-bad matrices, there are choices of $f$ and $g$ such that optimal rate of convergence is only $O(\ep)$ by Corollary~\ref{thm:main3}.

 A trivial example of a $c$-good matrix is the identity matrix $A\equiv I$. 
But does a $c$-bad matrix ever exist? 
A positive answer to this question would mean that $O(\ep)$ is the optimal rate in the general setting.
Furthermore, do we expect the majority of matrices in $C^{2,\al}\left(\T^n, \cS^n_+\right)$ to be good or bad for fixed $\al \in (0,1)$? 
To the best of our knowledge, these questions were not yet studied in the literature, and we view them as the main challenges in our paper.

To answer these questions, we
 let the topology of $C^{2,\al}\left(\T^n, \cS^n_+\right)$ be induced by the following metric
\[
d(A,B)=\|A-B\|_{C^{2,\al}} = \sum_{i,j=1}^n \|a_{ij} - b_{ij}\|_{C^{2,\al}(\T^n)},
\qquad \text{ for }  A, B \in C^{2,\al}\left(\T^n, \cS^n_+\right).
\]

As our second main result,  we show that the set of $c$-bad matrices ``dominates", confirming that an optimal rate of $O(\ep)$ should be expected for the ``majority" of matrices $A\in C^{2,\al}(\T^n, \cS^n_+)$.
\begin{thm}\label{thm:main4}
Assume $n\geq 2$.
The set of $c$-bad matrices is open and  dense in $C^{2,\al}\left(\T^n, \cS^n_+\right)$.
\end{thm}

\begin{rem}
Theorem \ref{thm:main1}, Corollary \ref{thm:main3}, and Theorem \ref{thm:main4} allow us to conclude that, generically, the optimal rate of convergence of $u^\ep-u$ to $0$ in ${L^p(U)}$ is also $O(\ep)$ for any given $p\geq 1$.
\end{rem}

\medskip

We next give several important cases where situation (i) of Theorem \ref{thm:main1} occurs.

\begin{thm}\label{thm:main2}
If one of the following points happens
\begin{itemize}
\item[(a)]\label{item:a-ep^2} $u$ is quadratic in $U$, that is, $D^2u$ is a constant matrix in $U$;

\item[(b)]\label{item:b-ep^2} $(a_{ij}(y) r(y))_{y_i}=0$ for all $1\leq j \leq n$, and $y\in \T^n$;

\item[(c)] $A$ is a shifted even function, namely, there exists $x\in\T^n$ such that $A(x-y)=A(x+y)$ for all $y \in \T^n$;
\end{itemize}
then situation {\rm (i)} of Theorem \ref{thm:main1} holds true.
\end{thm}

\begin{rem}
Let us discuss condition (b) of Theorem \ref{thm:main2} here.
Firstly, it is clear that if $(a_{ij}(y) r(y))_{y_i}=0$ for all $1\leq j \leq n$, and $y\in \T^n$, then $c^{kl}_j=0$ for all $1\leq j, k, l \leq n$.

Secondly, it is worth noting that the terms $(a_{ij}(y) r(y))_{y_i}$ for all $1\leq j \leq n$ were already discussed in Avellaneda, Lin \cite{AvL}.
In \cite{AvL}, it was denoted by
\[
b_j(y)= - (a_{ij}(y) r(y))_{y_i}.
\]
Under condition (b), we are able to write our non-divergence form operator $-a_{ij}(y) \phi_{y_i y_j}$ in divergence form by using the invariant measure $r$ as
\[
-r(y) a_{ij}(y) \phi_{y_i y_j} = - \left( r(y) a_{ij}(y) \phi_{y_j}\right)_{y_i} \qquad \text{ for all } \phi \in C^2(\T^n).
\]
\end{rem}

\begin{cor}\label{cor:main2}
We have situation {\rm (b)} in Theorem \ref{thm:main2} if one of the following conditions holds true
\begin{itemize}
\item $A(y)= a(y) I_n$ for some given $a \in C^2(\T^n, (0,\infty))$;

\item $A(y) = {\rm diag}\{a_1(y_1), a_2(y_2),\ldots, a_n(y_n)\}$ for some $a_i \in C^2(\T, (0,\infty))$ for all $1\leq i \leq n$;

\item $A(y) = {\rm diag}\{a_1(y), a_2(y),\ldots, a_n(y)\}$ for some $a_i \in C^2(\T^n, (0,\infty))$ such that $a_i$ is independent of $y_i$ for all $1\leq i \leq n$;

\item $A(y)=A(y_1)$, that is, $A$ depends only on $y_1$, and $a_{1j}=0$ for $2\leq j \leq n$.

\end{itemize}
\end{cor}
In particular, in dimension $n=1$, $A(y)=a(y)$, and so all matrices in  $C^2\left(\T, \cS_+\right)$ are $c$-good.
The specific cases of $A$ discussed in the above corollary match exactly with the discussions and the numerical results in 
Froese, Oberman \cite{FrO} (see examples on layered materials therein).
We only list some representative cases of $A$ in Corollary~\ref{cor:main2}, and one can come up with other similar examples of these types.

\subsection{Organization of the paper}
In Section \ref{sec:prelim}, we prove Theorems \ref{thm:prelim} and \ref{thm:main1}, and also Corollary \ref{thm:main3}.
Analysis on $c$-bad matrices and the proof of Theorem \ref{thm:main4} are provided in Section \ref{sec:good-bad}.
In particular, we show that the set of $c$-bad matrices is nonempty and construct some explicit examples.
We  give proofs of Theorem \ref{thm:main2} and Corollary~\ref{cor:main2} in Section \ref{sec:square rate}.

\subsection*{Notations}
The flat $n$-dimensional torus is denoted by $\T^n= \R^n/\Z^n$.
For $y \in \T^n$, we write $y=(y_1,y_2,\ldots, y_n)$.
Let $\cS^n_+$ be the set of all real symmetric, positive definite matrices of size $n$.
Denote by $I_n$ the identity matrix of size $n$.

\subsection*{Acknowledgement}
We would like to thank Fanghua Lin for some very useful discussions.
We thank Timo Sprekeler for pointing out some typos and imprecise points in the published version of the paper.

\section{Proofs of Theorems \ref{thm:prelim} and \ref{thm:main1}} \label{sec:prelim}

\begin{proof}[Proof of Theorem \ref{thm:prelim}]
Set
\[
\phi^\ep(x) = u(x) + \ep^2 v\left(\tfrac{x}{\ep},D^2 u(x)\right)=  u(x) + \ep^2 v^{kl}\left(\tfrac{x}{\ep}\right) u_{x_k x_l}(x) \qquad \text{ for all } x\in \ol U.
\]
Then,
\begin{align*}
&-a_{ij}\left(\tfrac{x}{\ep}\right)  \phi^\ep_{x_i x_j}(x)\\
=\ &-a_{ij}\left(\tfrac{x}{\ep}\right)\Big[ \left(  u_{x_i x_j}(x) +  v_{y_i y_j}^{kl}\left(\tfrac{x}{\ep}\right)  u_{x_k x_l}(x) \right) + \ep^2 v^{kl}\left(\tfrac{x}{\ep}\right)  u_{x_i x_j x_k x_l}(x) \\
&\qquad \qquad \qquad \qquad \qquad \qquad \qquad\qquad \qquad \qquad  + 2 \ep v^{kl}_{y_i}\left(\tfrac{x}{\ep}\right) u_{x_j x_k x_l} (x) \Big]\\
= \ & - \ol a_{ij} \,  u_{x_i x_j} (x)  +O( \ep^2) - 2 \ep a_{ij}\left(\tfrac{x}{\ep}\right) v^{kl}_{y_i}\left(\tfrac{x}{\ep}\right) u_{x_j x_k x_l} (x) \\
= \ & f(x) + O(\ep^2) + O(\ep).
\end{align*}
Then, by the usual maximum principle,
\[
\|u^\ep - \phi^\ep\|_{L^\infty(U)} \leq C\ep,
\]
which gives \eqref{sub-optimal}.
\end{proof}

\begin{rem}
From the proof above, the $O(\ep)$ rate comes from $a_{ij}\left(\frac{x}{\ep}\right) v^{kl}_{y_i}\left(\frac{x}{\ep}\right) u_{x_j x_k x_l} (x)$. Hence, in order to investigate the optimal rate, 
one needs to understand better the contribution of the source term $a_{ij}\left(\frac{x}{\ep}\right) v^{kl}_{y_i}\left(\frac{x}{\ep}\right) u_{x_j x_k x_l} (x)$ to \eqref{D-ep}.
\end{rem}

\begin{proof}[Proof of Theorem \ref{thm:main1}]
We consider the following equation
\begin{equation}\label{w-ep}
\begin{cases}
-a_{ij}\left(\frac{x}{\ep}\right) w^\ep_{x_i x_j}  =a_{ij}\left(\frac{x}{\ep}\right) v_{y_i}^{kl}\left(\frac{x}{\ep}\right) u_{x_j x_k x_l} (x) \qquad &\text{ in } U,\\
w^\ep = 0 \qquad &\text{ on } \partial U.
\end{cases}
\end{equation}
For $1\leq d, k, l \leq n$ fixed, denote by $p^{dkl} \in C(\T^n)$ a solution to
\[
-a_{ij}(y) p^{dkl}_{y_i y_j}(y) = a_{id}(y) v^{kl}_{y_i}(y) - c^{kl}_{d} \qquad \text{ in } \T^n.
\]
By \eqref{eq:def-c_ijk}, we can use the invariant measure $r$ to compute $c^{kl}_{d}$ as
\[
c^{kl}_{d} = \int_{\T^n} a_{id}(y) v^{kl}_{y_i}(y) r(y)\,dy = - \int_{\T^n} (a_{id}(y) r(y))_{y_i}v^{kl}(y)\,dy.
\]
Next, define
\[
\psi^\ep(x) = \ep^2 p^{jkl}\left(\tfrac{x}{\ep}\right) u_{x_j x_k x_l}(x) \qquad \text{ for all } x\in \ol U.
\]
Then,
\[
-a_{ij}\left(\tfrac{x}{\ep}\right) \psi^\ep_{x_i x_j}  =a_{ij}\left(\tfrac{x}{\ep}\right) v_{y_i}^{kl}\left(\tfrac{x}{\ep}\right) u_{x_j x_k x_l} (x) - h(x) + O(\ep).
\]
Let $z^\ep$ be the solution to
\begin{equation}\label{z-ep}
\begin{cases}
-a_{ij}\left(\frac{x}{\ep}\right) z^\ep_{x_i x_j}  =-h(x) \qquad &\text{ in } U,\\
z^\ep = 0 \qquad &\text{ on } \partial U.
\end{cases}
\end{equation}
By the maximum principle and Theorem \ref{thm:prelim}, we obtain the following estimates
\[
\|(w^\ep - \psi^\ep) + z^\ep\|_{L^\infty(U)} \leq C\ep,
\]
and (recall that  $z$ is the solution to \eqref{z-0})
\[
\|z^\ep - z\|_{L^\infty(U)} \leq C\ep.
\]
Thus,
\begin{equation}\label{est-imp-1}
\|(w^\ep - \psi^\ep) + z\|_{L^\infty(U)} \leq C\ep.
\end{equation}

\medskip

On the other hand, set $\varphi^\ep(x) = \phi^\ep(x) + 2\ep w^\ep(x)$. 
Then
\[
-a_{ij}\left(\frac{x}{\ep}\right) \varphi^\ep_{x_i x_j}  =f(x) + O(\ep^2) \qquad \text{ in } U,
\]
and $\|u^\ep - \varphi^\ep\|_{L^\infty(\partial U)} \leq C\ep^2$. 
Therefore, by the maximum principle again,
\begin{equation}\label{est-imp-2}
\|u^\ep - \varphi^\ep\|_{L^\infty(U)} \leq C\ep^2.
\end{equation}
Combine \eqref{est-imp-1} and \eqref{est-imp-2} to yield 
\[
\|u^\ep - u + 2 \ep z\|_{L^\infty(U)} \leq C\ep^2.
\]
We thus obtain \eqref{u-ep-u-z}.
If $h \equiv 0$, then $z \equiv 0$, and \eqref{u-ep-u-z} gives claim (i) right away. 
In particular, the optimal rate of convergence of $\|u^\ep - u \|_{L^\infty(U)}$ is $O(\ep^2)$.
Else, if $h \not\equiv 0$, then $z \not\equiv 0$, and claim (ii) holds with  optimal rate  $O(\ep)$.
\end{proof}

Finally, we give a proof of Corollary~\ref{thm:main3}.

\begin{proof}[Proof of Corollary~\ref{thm:main3}]
Assume that $C_{jkl} \neq 0$ for some fixed $1\leq j,k,l \leq n$.
Consider the equation \eqref{D-0} with 
\[
f(x) = -2 \left ( \ol a_{jk} x_l + \ol a_{kl} x_j + \ol a_{lj} x_k \right),
\quad
g(x) = x_j x_k x_l
 \qquad \text{ for all } x\in \ol U,
\]

Then, it is straightforward that the solution to \eqref{D-0} is
\[
u(x) = x_j x_k x_l \qquad \text{ for all } x \in \ol U.
\]
In particular, $h \equiv 2 C_{jkl} \neq 0$ in $U$.
\end{proof}


\section{Analysis on $c$-bad matrices} \label{sec:good-bad}

\subsection{The existence of $c$-bad matrices}

Let us first show that the set of $c$-bad matrices is not empty for $n\geq 2$.

\begin{prop}\label{prop:main3}
Assume $n\geq 2$.
The set of $c$-bad matrices is not empty.
\end{prop}

We will provide two different proofs of Proposition~\ref{prop:main3}
in dimension $n=2$.
The first proof relies on an explicit construction of $A \in C^2\left(\T^2, \cS^2_+\right)$ such that $c^{11}_1(A) \neq 0$,
whereas the second proof is via contradiction.
The corresponding proofs in higher dimensions are similar.

%

\begin{proof}[First proof of Proposition~\ref{prop:main3}]
We only consider the two dimensional case $n=2$.

\medskip

Let $A^0 (y)= \text{diag}\{1,\alpha(y)\}$ for some $\alpha \in C^\infty(\T^2, (0,\infty))$ such that 
\begin{equation}\label{condition-alpha}
(\log \alpha)_{y_1 y_2} \not \equiv 0.
\end{equation}
Denote by $r^0$ the corresponding invariant measure of $A^0$, that is,
\[
\begin{cases}
-r^0_{y_1 y_1} - (\alpha(y) r^0)_{y_2 y_2} = 0 \qquad \text{ in } \T^2,\\
r^0>0 \qquad \text{and} \qquad \int_{\T^2} r^0(y)\,dy=1.
\end{cases}
\]

First, we claim that $r^0_{y_1} \not \equiv 0$.
Indeed, assume otherwise that $r^0_{y_1} \equiv 0$, then
\[
(\alpha(y) r^0)_{y_2 y_2} = 0,
\]
which, together with the periodicity of $\alpha(y) r^0(y)$, implies 
\[
\alpha(y) r^0(y) = \phi(y_1)
\]
for some $1$-periodic function $\phi \in C^\infty(\T, (0,\infty))$.
Using the fact that $r^0_{y_1}=0$, we get
\[
0= \left(\frac{\phi(y_1)}{\alpha(y)} \right)_{y_1}= \frac{\phi'(y_1) \alpha(y) - \phi(y_1) \alpha_{y_1}(y)}{\alpha(y)^2}.
\]
Therefore, $(\log \phi(y_1))' = (\log \alpha(y))_{y_1}$, and hence,
\[
(\log \alpha)_{y_1 y_2} = 0,
\]
which contradicts \eqref{condition-alpha}.

Next, let $v(y) = s r^0_{y_1}(y)$ for some $s>0$ sufficiently small such that
\[
|v_{y_1 y_1}(y) + \alpha(y) v_{y_2 y_2}(y)| = s |r^0_{y_1 y_1 y_1}(y) + \alpha(y) r^0_{y_1 y_2 y_2}(y)| \leq \tfrac{1}{2} \qquad \text{ for } y \in \T^2.
\]
We let  $A(y) = \text{diag}\{a_1(y),a_2(y)\}$ be the matrix with
\[
a_1(y) = \left[1+v_{y_1 y_1}(y) + \alpha(y) v_{y_2 y_2}(y)\right]^{-1},
\quad 
a_2(y) = \alpha(y) a_1(y)  \qquad \text{ for } y \in \T^2.
\]
By the choice of $v$, it is clear that $\frac{1}{2} \leq a_1 \leq 2$.
Using the formula of $a_1, a_2$, one has
\[
-a_1(y) v_{y_1 y_1} - a_2(y) v_{y_2 y_2} - a_1(y) = -1 \qquad \text{ in } \T^2.
\]
Hence $v$ solves the cell problem \eqref{cell-kl} for $k=l=1$.
Therefore, $v=v^{11}$ and $\ol a_{11} =1$.

Let $r$ be the invariant measure corresponding to 
$A = a_1(y) A^0$. 
Note that
\[
r(y)  = \frac{r^0(y)}{a_1(y)}=  \left[1+v_{y_1 y_1}(y) + \alpha(y) v_{y_2 y_2}(y)\right] r^0(y)  \qquad \text{ for } y\in \T^2.
\]
Therefore,
\[
c^{11}_1 = \int_{\T^2} a_1(y) r(y) v_{y_1}(y)\,dy
=s \int_{\T^2} r^0(y) r^0_{y_1 y_1}(y)\,dy = -s \int_{\T^2} (r^0_{y_1})^2\,dy \neq 0.
\]
\end{proof}

The second approach is based on an asymptotic expansion at infinity. 
We aim at understanding deeper about the invariant measure and derive various consequences.
Consider a family of matrices $\{A^s\}$ indexed by $s>0$ of the form
\[
A^s(y)=\diag\{a_{1s}(y),s a_{2s}(y)\} \qquad \text{ for } y \in \T^2,
\]
 where $a_{1s}, a_{2s} \in C(\T^2)$, and there exists $C>0$ such that
\[
\frac{1}{C} \leq a_{1s}, a_{2s} \leq C.
\]
Let $r^s$ be the invariant measure of $A^s$. That is, $r^s$ solves
\begin{equation}\label{eq:r-s}
\begin{cases}
-(a_{1s}(y) r^s(y))_{y_1 y_1} - s (a_{2s} r^s(y))_{y_2 y_2} = 0 \qquad \text{ in } \T^2,\\
r^s>0 \qquad \text{and} \qquad \int_{\T^2} r^s(y)\,dy=1.
\end{cases}
\end{equation}
To simplify our notions a bit, let $v^{1s}$ be a solution to the cell problem
\[
-a_{1s}(y) v^{1s}_{y_1 y_1}-sa_{2s}(y) v^{1s}_{y_2 y_2} - a_{1s}(y) = - \ol{a}_{1s} \qquad \text{ in } \T^2.
\]
We now want to study the asymptotic of $r^s$ as $s \to \infty$.

\begin{thm}\label{thm:r-s}
Assume that $a_{1s} \to a_1$, $a_{2s} \to a_2$ uniformly in $\T^2$ for some $a_1, a_2 \in C(\T^2, [\frac{1}{C},C])$.
Then, $r^s \to r$ in $L^2(\T^2)$ as $s \to \infty$, where
\[
 r(y) = 
 \frac{B}{a_2(y) \left(\int_0^1 \frac{a_1(y)}{a_2(y)}\,dy_2\right)} \qquad \text{ for } y \in \T^2.
\]
Here, $B>0$ is a scaling constant so that $\int_{\T^2} r(y)\,dy=1$.
\end{thm}

\begin{proof}
 We divide the proof into a few steps.
 
 \medskip
 
\noindent {\bf Step 1.}
We first obtain a priori estimates for $r^s$.
Multiply \eqref{eq:r-s} by $r^s$ and integrate by parts to get
\begin{align*}
0&= \int_{\T^2} (a_{1s} r^s)_{x_1} r^s_{x_1} + s (a_{2s} r^s)_{x_2} r^s_{x_2}\,dx\\
&= \int_{\T^2} a_{1s} (r^s_{x_1})^2 + s a_{2s} (r^s_{x_2})^2\,dx +\int_{\T^2} (a_{1s})_{x_1} r^s r^s_{x_1} + s (a_{2s})_{x_2} r^s r^s_{x_2}\,dx\\
&\geq \frac{1}{2C}  \int_{\T^2}  (r^s_{x_1})^2 + s  (r^s_{x_2})^2\,dx - C(1+s) \int_{\T^2} (r^s)^2\,dx.
\end{align*}
In particular,
\begin{equation}\label{r-s-1}
  \int_{\T^2}  (r^s_{x_2})^2\,dx\leq C \int_{\T^2} (r^s)^2\,dx.
\end{equation}
Next, multiply \eqref{eq:r-s} by $a_{2s} r^s$ and integrate by parts and do the estimates in the similar fashion as above to yield
\begin{equation}\label{r-s-2}
  \int_{\T^2}  (r^s_{x_1})^2 + s ((a_{2s} r^s)_{x_2})^2 \,dx\leq C \int_{\T^2} (r^s)^2\,dx.
\end{equation}
Combine \eqref{r-s-1} and \eqref{r-s-2} to deduce that
\begin{equation}\label{r-s-3}
  \int_{\T^2}  |Dr^s|^2 + s ((a_{2s} r^s)_{x_2})^2 \,dx\leq C \int_{\T^2} (r^s)^2\,dx,
\end{equation}
and in particular,
\begin{equation}\label{r-s-4}
\|Dr^s\|_{L^2} \leq C \|r^s\|_{L^2}.
\end{equation}
Note that $2^*=\infty$. 
By Sobolev's inequalities and \eqref{r-s-4},
\[
\|r^s\|_{L^3} \leq C(\|r^s\|_{L^2} +\|Dr^s\|_{L^2}) \leq C \|r^s\|_{L^2}.
\]
On the other hand, by H\"older's inequality and the fact that $\int_{\T^2} r^s\,dx=1$, we have
\[
\|r^s\|_{L^3}^3 = \left(\int_{\T^2} (r^s)^3\,dx \right)\left(\int_{\T^2}  r^s\,dx \right) \geq \left(\int_{\T^2} (r^s)^2\,dx \right)^2 =\|r^s\|_{L^2}^4.
\]
Combine the two inequalities above to get $\|r^s\|_{L^2} \leq C$.
This, together with \eqref{r-s-3}, implies 
\begin{equation}\label{r-s-5}
 \|r^s\|_{L^2} +\|Dr^s\|_{L^2} + s \|(a_{2s} r^s)_{x_2}\|_{L^2}\leq C.
\end{equation}

 \medskip
 
\noindent {\bf Step 2.}
By compactness, by passing to a subsequence if needed as $s \to \infty$, we have
\[
\begin{cases}
r^s \to r \qquad &\text{ in } L^2(\T^2),\\
Dr^s \rightharpoonup D r \qquad &\text{ weakly in } L^2(\T^2),\\
(a_2 r)_{x_2}=0.
\end{cases}
\]
Thus, we have that
\[
r(x) = \frac{\xi(x_1)}{a_2(x)} \quad \text{ and } \quad \int_{\T^2}  \frac{\xi(x_1)}{a_2(x)}\,dx=1.
\]
Here, $\xi$ is a periodic function and $\xi \geq 0$ in $\T$.
We aim at characterizing $\xi$ better.
Let $\phi = \phi(x_1) \in C^\infty(\T)$ be a test function.
Multiply \eqref{eq:r-s} with $\phi$ and integrate by parts to have
\[
\int_{\T^2} a_{1s}(x) r^s(x)\phi''(x_1)\,dx_2 dx_1=0.
\]
Let $s\to \infty$ and use the formula of $r$ to yield
\[
\int_0^1 \left (\int_0^1 \frac{a_1(x)}{a_2(x)}\,dx_2\right) \xi(x_1) \phi''(x_1)\,dx_1=0.
\]
Therefore,
\[
\left( \left (\int_0^1 \frac{a_1(x)}{a_2(x)}\,dx_2\right) \xi(x_1) \right)_{x_1 x_1} =0,
\]
which means
\[
 \left (\int_0^1 \frac{a_1(x)}{a_2(x)}\,dx_2\right) \xi(x_1) = A x_1+ B,
\]
for some constants $A,B \in \R$.
As the left hand side above is an $1$-periodic function, we get that $A=0$, which implies
\[
r(x) = \frac{B}{a_2(x) \left(\int_0^1 \frac{a_1(x)}{a_2(x)}\,dx_2\right)} \qquad \text{ for } x\in \T^2.
\]
Again, $B>0$ is simply a scaling constant so that $\int_{\T^2} r(x)\,dx=1$.
\end{proof}

Here is another way of showing that  $c^{11}_1(\cdot) \not\equiv 0$ in $C^2\left(\T^2, \cS^2_+\right)$.
This proof is indirect.

\begin{proof}[Second proof of Proposition \ref{prop:main3}]
We prove by contradiction.
Assume that the proposition fails, then
 $c^{11}_1(A)=0$ for all $A  \in C^2\left(\T^2, \cS^2_+\right)$.

\smallskip

Let $\theta \in C^\infty(\T^2, [1,2])$ and $\psi \in C^\infty(\T^2)$ be two functions such that
\begin{equation}\label{non-0}
\int_{\T^2} \tfrac{\theta(x)}{\int_0^1 \theta(x)\,dx_2} \psi_{x_1}\,dx \neq 0.
\end{equation}
Let $C_0 = 2\left( \|\theta\|_{L^\infty} +1\right) \|D^2\psi\|_{L^\infty}$,  and set
\[
v^{1s}(x) := \frac{\psi(x)}{C_0 s}.
\]
Consider a family of matrices $A^s=\diag\{a_{1s},sa_{2s}\}$, $s>1$, with
\[
a_{2s}  = 
\left[\theta(x)+\left(\theta(x)v^{1s}_{x_1 x_1} + s v^{1s}_{x_2 x_2} \right)\right]^{-1}
=\left[\theta(x)+\left(\tfrac{\theta(x) \psi_{x_1 x_1}}{C_0 s} + \tfrac{ \psi_{x_2 x_2}}{C_0} \right)\right]^{-1}
\]
and $a_{1s} := \theta(x) a_{2s}$. By the definition of $C_0$, we have  $\frac{1}{3} \leq a_{2s} \leq 2$.
Moreover, $a_{2s} \to a_2:= \left[\theta(x)+  \frac{ \psi_{x_2 x_2}}{C_0}\right]^{-1}$ and $a_{1s} \to a_1:=\theta(x) a_2$ uniformly in $\T^2$ as $s\to \infty$. By Theorem~\ref{thm:r-s}, as $s\to\infty$, the invariant measure $r^s$ of $A^s$ converges in $L^2(\T^2)$ to 
\[
r(x)=\frac{B}{a_2(x)\int_0^1\theta(x)\,d x_2}, \qquad \text{ for } x \in\T^2.
\]

 Further, observe that $v^{1s}$ solves the $(1,1)$-th cell problem (cf. \eqref{cell-kl}) of $A^s$:
\[
-a_{1s}(x) v^{1s}_{x_1 x_1}-sa_{2s}(x) v^{1s}_{x_2 x_2} - a_{1s}(x) = - \ol{a}_{1s}:=-1 \qquad \text{ in } \T^2.
\]
By our assumption, $c_1^{11}(A^s)=0$ for all $s>1$. Hence, for all $s>1$,
\[
0=C_0 sc_1^{11}(A^s)= \int_{\T^2} a_{1s}(x) r^{s}(x) \psi_{x_1}(x)\,dx.
\]
Therefore,
\[
0= \lim_{s\to \infty}  \int_{\T^2} a_{1s}(x) r^{s}(x)\psi_{x_1}(x)\,dx\\
= \int_{\T^2} a_{1}(x) r(x)\psi_{x_1}(x)\,dx = B \int_{\T^2} \tfrac{\theta(x)}{\int_0^1 \theta(x)\,dx_2} \psi_{x_1}\,dx,
\]
which contradicts \eqref{non-0}.
\end{proof}

\subsection{The set of $c$-bad matrices is open and  dense in $C^{2,\al}\left(\T^n, \cS^n_+\right)$}

We first show that $A \mapsto c^{kl}_j(A)$ is continuous in $C^{2,\al}\left(\T^n, \cS^n_+\right)$.

\begin{lem}\label{lem:c-cont}
We have that $A \mapsto c^{kl}_j(A)$ is continuous in $C^{2,\al}\left(\T^n, \cS^n_+\right)$.
\end{lem}

\begin{proof}[Sketch of proof]
Take a sequence $\{A^m\} \subset C^{2,\al}\left(\T^n, \cS^n_+\right)$ such that $\|A^m-A\|_{C^{2,\al}} \to 0$ as $m \to \infty$ for some $A\in C^{2,\al}\left(\T^n, \cS^n_+\right)$.
Write $A^m(y)=(a_{ij}^m(y))_{1 \leq i, j \leq n}$ and their corresponding effective matrices as $\ol {A^m} = (\ol a_{ij}^m)_{1 \leq i,j \leq n}$.

Fix $1 \leq k,l \leq m$.
For each $m \in \N$, the corresponding cell problem is
\[
-a_{ij}^m(y) v^m_{y_i y_j}(y) - a_{kl}^m(y) = -\ol a_{kl}^m \qquad \text{ in } \T^n.
\]
By subtracting to a constant, we suppose that $v^m(0)=0$.
By usual a priori estimates, for fixed $\al \in (0,1)$, there exists $C>0$ independent of $m$ such that
\[
\|v^m\|_{C^{2,\al}(\T^n)} \leq C\left(1+\|A^m\|_{C^{0,\al}}\right) \leq C.
\]
Therefore, it is not hard to see that $v^m \to v$ in $C^2(\T^n)$, where $v$ solves
\[
-a_{ij}(y) v_{y_i y_j}(y) - a_{kl}(y) = -\ol a_{kl} \qquad \text{ in } \T^n.
\]
Next, let $\{r^m\}$ and $r$ be the invariant measures corresponding to $\{A^m\}$ and $A$, respectively.
By repeating Step 1 of the proof of Theorem \ref{thm:r-s}, it is clear that there exists a constant $C>0$ independent of $m$ so that
\[
\|r^m\|_{L^2(\T^n)} + \|Dr^m\|_{L^2(\T^n)} \leq C.
\]
By compactness, we also see that $r^m \to r$ in $L^2(\T^n)$.
Thus,
\[
\lim_{m \to \infty} c^{kl}_{j}(A^m) =\lim_{m \to \infty} \int_{\T^n} a_{ij}^m(y) v^{m}_{y_i}(y) r^m(y)\,dy = \int_{\T^n} a_{ij}(y) v_{y_i}(y) r(y)\,dy = c^{kl}_{j} (A).
\]
\end{proof}

We now provide a proof of Theorem \ref{thm:main4}.
Note that Proposition \ref{prop:main3} is not needed here in the proof, but some key ideas in its proof are used essentially.

\begin{proof}[Proof of Theorem \ref{thm:main4}]
As $A \mapsto C_{jkl}(A)$ is continuous in $C^{2,\al}\left(\T^n, \cS^n_+\right)$, the set of $c$-bad matrices is open in $C^{2,\al}\left(\T^n, \cS^n_+\right)$.
We therefore only need to show that this set is dense.

\medskip

Fix a $c$-good matrix $A^0\in C^{2,\al}\left(\T^n, \cS^n_+\right)$ and $\delta>0$.
Our aim is to show the existence of a $c$-bad matrix $A\in C^{2,\al}\left(\T^n, \cS^n_+\right)$ such that $\|A-A^0\|_{C^{2,\al}} \leq \delta$.
Let $r^0$ be the invariant measure corresponding to $A^0$.

\medskip

\noindent {\bf Step 1.}
 We first aim at finding $A^1 \in C^{2,\al}\left(\T^n, \cS^n_+\right)$  such that $\|A^1-A^0\|_{C^{2,\al}} \leq \frac{\delta}{2}$ and 
\[
(a_{ij}^1(y) r^1(y))_{y_i} \not \equiv 0
\]
 for some $1 \leq j \leq n$.
 Here $r^1$ denotes the invariant measure of $A^1$. Of course, if $(a_{ij}^0(y) r^0(y))_{y_i} \not \equiv 0$ for some $j\in \{1,2,\ldots,n\}$, then we simply let $A^1=A^0$.
Otherwise,
\[
(a_{ij}^0(y) r^0(y))_{y_i} = 0 \qquad \text{ for all } 1\leq j \leq n.
\]
In this case, we take $\xi \in C^\infty(\T)$ with $\xi' \not \equiv 0$, and define $A^1=(a_{ij}^1)_{1\le i,j\le n}$ as 
 \[
 a_{ij}^1(y)=
 \begin{cases}
 a_{11}^0(y) + \frac{\delta \xi(y_1+y_2)}{r^0(y)} \qquad &\text{ for } i=j=1,\\
  a_{22}^0(y) - \frac{\delta \xi(y_1+y_2)}{r^0(y)} \qquad &\text{ for } i=j=2,\\
  a_{ij}^0(y) \qquad &\text{ otherwise}.
 \end{cases}
 \]
  Choosing $\xi$ with $\|\xi\|_{C^{2,\al}(\T)}$ sufficiently small, we have $\|A^1-A^0\|_{C^{2,\al}} \leq \frac{\delta}{2}$.
   It is not hard to see that $r^1=r^0$ as
 \[
 -(a^1_{ij}(y) r^0(y))_{y_i y_j} = \delta \left[(\xi(y_1+y_2))_{y_2 y_2}-(\xi(y_1+y_2))_{y_1 y_1} \right] =0.
 \]
  Moreover,
 \[
(a_{i1}^1(y) r^1(y))_{y_i} = \delta \xi'(y_1+y_2)  \not \equiv 0.
\]

\noindent Step 1 is complete.

\medskip

\noindent {\bf Step 2.}
Next, we will find a $c$-bad matrix $A$ such that $\|A^1-A\|_{C^{2,\al}} \leq \frac{\delta}{2}$. 
Of course, if $c^{11}_1(A^1) \neq 0$, then $A^1$ is $c$-bad and we simply put $A=A^1$.
 If $c^{11}_1(A^1)=0$, then, using the same idea as in the proof of Proposition \ref{prop:main3}, we will construct $A$ as follows.
By the construction of $A^1$, without loss of generality, we assume
\[
(a_{i1}^1(y) r^1(y))_{y_i} \not \equiv 0.
\] 

Let $\phi(y) = s (a_{i1}^1(y) r^1(y))_{y_i}$, $\gam(y)= \left[1 + a_{ij}^1(y) \phi_{y_i y_j}(y) \right]^{-1}$, and set
\[
A(y) = \gam(y) A^1(y),
\]
where $s>0$ is chosen to be small enough so that $\|A^1-A\|_{C^{2,\al}} \leq \tfrac{\delta}{2}$. 
Note that the invariant measure $r(y)$ of $A(y)=\gamma(y) A^1(y)$ is
\[
r(y)  =\frac{r^1(y)}{\gam(y)} =\left[1 + a_{ij}^1(y) \phi_{y_i y_j}(y) \right] r^1(y)  \qquad \text{ for } y\in \T^n.
\]

It remains to show that $A$ is $c$-bad. To this end, observe that
\[
-a_{ij} \phi_{y_i y_j} - \gam
=-\gam (a_{ij}^1\phi_{y_iy_j}+1)
 = -1 \qquad \text{ in } \T^n.
\]
Recall that $v^{11}$ solves
the $(1,1)$-th cell problem \eqref{cell-kl} for $A^1$:
 \[
 -a_{ij}^1(y) v^{11}_{y_i y_j}(y) - a_{11}^1(y) =- \ol a_{11}^1 \qquad \text{ in } \T^n.
 \]
Let $v(y)=v^{11}(y) + \ol a_{11}^1 \phi(y)$.
Thanks to the above identities,
\[
 -a_{ij}(y) v_{y_i y_j}(y) - a_{11}(y) =- \ol a_{11}^1 \qquad \text{ in } \T^n.
\]
That is, $v$ solves the $(1,1)$-th cell problem \eqref{cell-kl} for $A$.
Thus,
\begin{align*}
c^{11}_1(A) &=  \int_{\T^n} a_{i1}(y) r(y) v_{y_i}(y)\,dy\\
&=\int_{\T^n} a_{i1}^1(y) r^1(y) v_{y_i}(y)\,dy = - \int_{\T^n} (a_{i1}^1(y) r^1(y))_{y_i} (v^{11}(y)+ \ol a_{11}^1 \phi(y))\,dy\\
& = c^{11}_1(A^1) - \ol a_{11}^1 \int_{\T^n} (a_{i1}^1(y) r^1(y))_{y_i} \phi(y)\,dy\\
&=- s \ol a_{11}^1 \int_{\T^n} \left((a_{i1}^1(y) r^1(y))_{y_i} \right)^2\,dy\neq 0.
\end{align*}
Therefore, $A$ is $c$-bad and $\|A-A^0\|_{C^{2,\al}} \leq \delta$. 
The proof is complete.
\end{proof}
We note that in the above proof, we can choose $\phi(y)=s q(y)$ for any $q\in C^\infty(\T^n)$ such that
\[
\int_{\T^n} (a_{i1}^1(y) r^1(y))_{y_i} q(y) \,dy \neq 0.
\]


\section{Optimal rate of convergence $O(\ep^2)$} \label{sec:square rate}
In this section, we discuss the situations where the optimal rate of convergence of $\|u^\ep - u \|_{L^\infty(U)}$ is $O(\ep^2)$.
Let us give proofs of Theorem \ref{thm:main2} and Corollary~\ref{cor:main2}.

\begin{proof}[Proof of Theorem \ref{thm:main2}]
We only need to show that $h \equiv 0$ in all situations.
In the first situation, $D^2 u $ is a constant matrix, then clearly $D^3 u=0$ in $U$, and hence $h \equiv 0$.

\smallskip

In the second situation, $(a_{ij}(y) r(y))_{y_i}=0$ for all $1\leq j \leq n$, and $y\in \T^n$.
Then
\[
c^{kl}_{j} = \int_{\T^n} a_{ij}(y) v^{kl}_{y_i}(y) r(y)\,dy =   - \int_{\T^n} (a_{ij}(y) r(y))_{y_i}(y) v^{kl}(y) \,dy=0,
\]
which implies $h \equiv 0$.

\smallskip

Finally, in the last situation, we assume $x=0$ without loss of generality.
As $A$ is even, we see that $v^{kl}$ and $r$ are also even for $1\leq k,l \leq n$.
Then $v^{kl}_{y_i}$ is odd, that is, $v^{kl}_{y_i}(y) = -v^{kl}_{y_i}(-y)$.
Hence, $y \mapsto a_{ij}(y) v^{kl}_{y_i}(y) r(y)$ is odd as well, which gives that
\[
c^{kl}_{j} = \int_{\T^n} a_{ij}(y) v^{kl}_{y_i}(y) r(y)\,dy=0.
\]
\end{proof}

\begin{proof}[Proof of Corollary~\ref{cor:main2}]
We aim at showing $(a_{ij}(y) r(y))_{y_i}=0$ for all $1\leq j \leq n$, and $y\in \T^n$ in all cases.

\smallskip

In the first case, $A(y)= a(y) I_n$ for some $a \in C^2(\T^n, (0,\infty))$.
Then, the invariant measure $r$ is simply $r(y) = c/a(y)$ for $y\in \T^n$, where 
$c=\left[\int_{\T^n} 1/a(y)\,dy\right]^{-1}$.
It is clear then that $(a_{ij}(y) r(y))_{y_i}=0$ for $1 \leq  j \leq n$.

\smallskip

In the second case,  we have $A(y) = {\rm diag}\{a_1(y_1), a_2(y_2),\ldots, a_n(y_n)\}$ for some $a_i \in C^2(\T, (0,\infty))$, $1\leq i \leq n$. 
The invariant measure of $A$ is
\[
r(y) = \frac{c}{a_1(y_1) a_2(y_2)\ldots a_n(y_n)}
\]
where $c>0$ is a normalization constant such that $\int_{\T^n} r(y)\,dy=1$.
We again get that $(a_{ij}(y) r(y))_{y_i}=0$.

\smallskip

Thirdly, we consider the situation where $A(y) = {\rm diag}\{a_1(y), a_2(y),\ldots, a_n(y)\}$ for some $a_i \in C^2(\T^n, (0,\infty))$ such that $a_i$ is independent of $y_i$ for all $1\leq i \leq n$.
This case is even more straightforward as $r \equiv 1$.

\smallskip

Lastly, in the fourth situation where $A(y)=A(y_1)$ with $a_{1j}=0$ for $2 \leq j \leq n$, it is not hard to see here that 
\[
r(y)=\frac{c}{a_{11}(y)} = \frac{c}{a_{11}(y_1)}
\]
 where $c>0$ is a normalization constant such that $\int_{\T^n} r(y)\,dy=1$.
 It is clear then that $(a_{ij}(y) r(y))_{y_i}=0$ for $1 \leq  j \leq n$.
\end{proof}

\end{document}